\documentclass[12pt,a4paper]{amsart} 
 
\usepackage[psamsfonts]{amssymb} 
\usepackage{amsfonts,amsmath} 

\usepackage{epsfig,multicol} 

\newtheorem{letterthm}{Theorem}

\newtheorem{theorem}{Theorem}[section]
\newtheorem{proposition}[theorem]{Proposition}
\newtheorem{lemma}[theorem]{Lemma}
\newtheorem{corollary}[theorem] {Corollary}

\def\virt{{\rm{virt}}}
\def\ker{{\rm{ker}}\ }

\def\IR{{\mathbb R}} 
\def\R{{\mathbb R}} 
 
\def\IS{{\mathbb S}} 
\def\S{{\mathbb S}} 
\def\IZ{{\mathbb Z}} 
\def\Z{{\mathbb Z}} 
 
\def\D{{\mathbb D}} 
 
\def\IQ{{\mathbb Q}}

\def\G{\Gamma}

\def\cd{\rm cd}
 
\newtheorem{remark}[theorem]{Remark}

\title[Cofinitely Hopfian groups]{Cofinitely Hopfian groups,
open mappings and knot complements} 
\author{M.R.~Bridson \and D.~Groves \and J.A.~Hillman \and G.J.~Martin}
\thanks{This paper is based upon work supported by the 
National Science Foundation (USA) and the Marsden Fund (NZ). Bridson
was supported by a Senior Fellowship from the EPSRC (UK) and a Royal
Society Wolfson Research Merit Award.
Hillman was the Grey College Mathematics Fellow for
Michaelmas Term of 2008 at Durham University.}



\begin{document} 
 
 
\begin{abstract} A group $\Gamma$ is defined to be {\em{cofinitely Hopfian}} 
if every homomorphism $\Gamma\to\Gamma$ whose image is of finite index 
is an automorphism. 
Geometrically significant groups 
enjoying this property include certain relatively hyperbolic groups and
many lattices.
A knot group is cofinitely Hopfian if and only if the knot is not a torus knot.
A free-by-cyclic group is cofinitely Hopfian if and only if 
it has trivial centre. 
Applications to the theory of  open 
mappings between manifolds are presented.
\end{abstract}

\maketitle 

\section{Introduction}

A group $\G$ is said to be {\it Hopfian} (or to have the Hopf property)
if every surjective endomorphism of $\G$ is an automorphism.
It is said to be (finitely) {\it co-Hopfian\/} if every injective endomorphism 
(with image of finite index) is an automorphism. These properties
came to prominence in
Hopf's work on self-maps of surfaces
\cite{Ho}. Related notions have since played a significant role in many  contexts,  
for instance combinatorial group theory,  e.g. \cite{BS,Sela1},  
and the study of  approximate fibrations \cite{Da}. 
Here we will be concerned with a Hopf-type property  that played a
central
role in the work of Bridson, Hinkkanen and Martin
on open and quasi-regular mappings of
compact negatively curved manifolds  \cite{BHM}.

We say that a group $\Gamma$ 
is {\em cofinitely Hopfian} (or has
the cofinite Hopf property)
if every homomorphism $\Gamma \to \Gamma$
whose  image is  of finite index   
is an automorphism.
This condition is stronger than both the Hopf property and the finite co-Hopf property. 
In Section \ref{s:defs} we discuss its relationship to other Hopf-type properties.
In Section \ref{s:top} we recall its relationship to open mappings
of manifolds.

If a torsion-free group is hyperbolic in the sense of Gromov, 
then it is cofinitely Hopfian: this can be deduced from
Zlil Sela's profound study of the endomorphisms of such groups  \cite{Sela1,Sela2}; see
\cite{BHM}.
In the present article, we shall extend this result to cover several other
classes of groups associated with non-positive curvature in group theory. 
Classical rigidity results imply that lattices 
in semisimple Lie groups with trivial centre and no compact
factors are cofinitely Hopfian (Section \ref{s:lattices}).  In the other classes that we consider
the issue is more subtle. These classes are the
toral relatively hyperbolic groups,
free-by-cyclic groups, and (classical) knot groups. 
The following are special cases of the
results that we prove in each case. 
 
\begin{letterthm}\label{t:trh} If a  toral relatively hyperbolic group is   co-Hopfian
then it is cofinitely Hopfian.
\end{letterthm}

\begin{letterthm}\label{t:F.Z} Let $F$ be a finitely generated free group and let $\G=F\rtimes_\theta\Z$.  The
following conditions are equivalent:
\begin{enumerate}
\item $\G$ is cofinitely Hopfian;
\item $\G$ has trivial centre;
\item $\theta$ has infinite order in ${\rm{Out}}(F)$.
\end{enumerate}
\end{letterthm}

\begin{letterthm}\label{t:knots} Let $K$ be a knot in $\S^3$. The
fundamental group of $\S^3\smallsetminus K$ is cofinitely Hopfian
if and only if $K$ is not a torus knot.
\end{letterthm}

If one can prove that the groups in a certain class are cofinitely Hopfian, one can immediately
deduce constraints on the nature  of open mappings of manifolds whose
fundamental groups lies in that class (see Section \ref{s:top}). 
To exemplify this, in Section \ref{s:knots}
we use Theorem \ref{t:knots} to prove:

\begin{letterthm} Let $K\subset\S^3$ be a knot. If $K$  is not a torus knot,
then every proper open self-mapping of $\S^3\smallsetminus K$ is homotopic 
to a homeomorphism.
\end{letterthm}

\section{Topological Motivation}\label{s:top}

The work of John Walsh \cite[Theorem 4.1]{Walsh}  and Steven Smale \cite{SS}
establishes a strong connection between the cofinite Hopf property of groups and open mappings of
quite general spaces. For simplicity we state their results only for manifolds, which is sufficient
for our purposes. Recall that a map is {\em proper} if the preimage of each compact set is compact, and
a {\em light} mapping is one for which the preimage of every point is a totally disconnected space,
for instance  a Cantor set.

\begin{proposition}\label{Walsh} Let $M_1$ and $M_2$ be connected manifolds (possibly
with boundary). If a map $f:M_1\to M_2$ is  proper, open and
surjective, then the index of $f_*\pi_1M_1$ in $\pi_1M_2$ is finite.
Moreover, $f$ induces a surjection on rational homology. 
\end{proposition} 

In the opposite direction we have:

\begin{theorem}\label{Walsh2}
If $M$ and $N$ are compact connected PL manifolds and $f:M\to N$ is
a map with $[\pi_1N:f_*\pi_1M]<\infty$,  
then $f$ is homotopic to a light open mapping.
\end{theorem}
 
In \cite{BHM}  these  results  were used in tandem with the fact that
torsion-free hyperbolic groups are cofinitely Hopfian to study  discrete open mappings
 between manifolds  of negative curvature. Motivation for this study
comes from the Lichnerowicz problem of 
identifying those $n$--manifolds which admit rational endomorphisms 
 \cite{MMP}, and from a desire to extend the scope of Mostow rigidity.

\section{Toral Relatively Hyperbolic Groups}
 
 A finitely generated group $\Gamma$ is {\em toral relatively hyperbolic}
if it is torsion-free and hyperbolic relative to a collection
of free abelian subgroups.  
We refer the
reader to \cite{DG} for a more expansive
definition. For our purposes, the following
observations will be more useful  than the technicalites of the definition. 
 
The following families of groups are  toral relatively hyperbolic:
\begin{itemize}
\item Torsion-free (word) hyperbolic groups;
\item Fundamental groups of finite-volume or geometrically finite hyperbolic manifolds whose cusp
cross-sections are tori;
\item limit groups \cite{Se-limit} (otherwise known as fully residually free groups),
and more generally groups that act freely on $\IR^n$--trees 
(cf.~\cite{ali2, Dah, vincent}).
\end{itemize}
A great deal is known about toral relatively hyperbolic groups. In particular, 
the isomorphism problem is solvable in this class \cite{DG}.
They are all Hopfian \cite{MR-RH}.  Degenerate examples such
as $\mathbb Z$
are not finitely
co-Hopfian, but the following result shows that this is the only obstruction to being
cofinitely Hopfian.

\begin{theorem} \label{Injects}
Let $\Gamma$ be a toral relatively hyperbolic group and let $\varphi : \Gamma
\to \Gamma$ be a homomorphism so that $\varphi(\Gamma)$ has finite index
in $\Gamma$.  Then $\varphi$ is an injection.
\end{theorem}

\begin{proof}
In order to obtain a contradiction we suppose that $[\Gamma : \varphi(\Gamma)] < \infty$
and that $\varphi$ is not an injection.   
Since $\Gamma$ is torsion-free and each $\varphi^i(\Gamma)$ has finite index
in $\Gamma$, the groups $K_i = \ker(\varphi) \cap \varphi^i(\Gamma)$ are all
non-trivial.  Consequently, the kernels $\ker(\varphi) \subset \ker(\varphi^2) \subset \ldots$
are all distinct, so  there is a properly descending sequence of epimorphisms:
\[	\Gamma \twoheadrightarrow \varphi(\Gamma) \twoheadrightarrow
\varphi^2(\Gamma) \twoheadrightarrow \ldots .	\]
But \cite[Theorem 5.2]{MR-RH} tells us that such sequences do not exist. (To see that \cite[Theorem 5.2]{MR-RH} applies, note that each $\phi^i(\Gamma)$, being a finitely generated subgroup of $\Gamma$, is a $\Gamma$-limit group.)
This contradiction implies that $\varphi$ must
in fact be an injection.
\end{proof}

\begin{corollary} \label{c:coH}
Let $\Gamma$ be a co-Hopfian toral relatively hyperbolic group. 
Then $\Gamma$ is cofinitely Hopfian.
\end{corollary}

\section{Lattices} \label{s:lattices}

The discussions leading to this paper began with a desire to extend the results of
\cite{BHM} concerning convex cocompact Kleinian groups to 
non-uniform lattices and geometrically
finite groups. 
Corollary \ref{c:coH} represents progress in this direction since
it applies to  the fundamental groups of finite volume hyperbolic 
manifolds whose cusp cross-sections are tori. However, non-uniform
lattices whose cusp
cross-sections are more general flat manifolds need not be toral relatively
hyperbolic (though they will always have a subgroup of finite index that
is), so are not covered by Corollary \ref{c:coH}.
 To circumvent this difficulty we use the following result of Hirshon \cite{Hi}
instead of Theorem \ref{Injects} (cf.~\cite{BHM} Proposition 4.4).
Hirshon's argument is a variant on the standard argument for the Hopficity of 
residually finite groups.

\begin{lemma}\label{lem2}
Let $\Gamma$ be a finitely generated residually finite group 
with no non-trivial finite normal subgroup.
If $\varphi:\Gamma\to{\Gamma}$ is an endomorphism with image of finite index 
then $\varphi$ is a monomorphism.
\end{lemma}

Although our initial interest lay with lattices in ${\rm{SO}}(n,1)$, 
it does not seem any harder to deal with similar lattices in more general Lie groups. 
We would like to have given an elementary proof\footnote{we thank Nicolas Monod for
his thoughts on this} covering a class of lattices that
includes those in ${\rm{SO}}(n,1)$. However, we are unable to find a 
proof using less machinery than what follows. (Some alternative approaches
are discussed in the next section.)

\begin{proposition}\label{TRHVH}
Let $G$ be a connected semisimple Lie group with trivial centre 
and no compact factors, 
and let $\G$ be a lattice in $G$.
If either
\begin{enumerate}
\item $\G$ is cocompact or 
\item $\G$ is irreducible
\end{enumerate}
 then it is cofinitely Hopfian.
\end{proposition}
\begin{proof} 
Since $G$ is connected and centreless, the adjoint representation 
is faithful and $G$ is linear.  
Also, $\G$ is finitely generated (see, for example, 
\cite[IX.3.1.(ii), p.311]{Margulis}),
and Mal'cev famously observed that finitely generated linear groups 
are residually finite \cite{Malcev}. 
Moreover $\G$ has no non-trivial finite normal subgroup 
(this follows from Borel's Density Theorem;
see \cite[Corollary 4.42]{Witte-Morris}). 
Thus we are in the situation of Lemma \ref{lem2}:
if the image of $\varphi : \Gamma \to \Gamma$ has finite index then
$\varphi$ is injective. 

The group $\G$ has a torsion-free subgroup $H$ of finite index,
by Selberg's Lemma (\cite{Sel}; 
see also \cite[Theorem 4.60]{Witte-Morris}).
Let $H_k=\varphi^{-k}(H)$, for $k\geq0$.
Then $[\G:H_k]\leq[\G:H]$ and $H_k$ is torsion-free, 
since $\varphi$ is injective, for all $k\geq0$.
Since $\G$ is finitely generated it has only finitely many subgroups of 
index bounded by $[\G:H]$.
Therefore there are integers $m,n>0$ such that $H_{m+n}=H_n$,
and hence $\varphi^m(H_n)\leq{H_n}$.
If $\G$ is cocompact then so is $H_n$ and in this case
$\varphi^m(H_n)=H_n$, by Theorem 2 of \cite{Re} (or by using any suitable
volume-type invariant, in particular L\"oh and Sauer's Lipschitz simplicial volume \cite{LoS}).
Since $\G$ has no non-trivial finite normal subgroup,
it follows that $\varphi^m(\G)=\G$ also,
and hence $\varphi$ is an automorphism.

If $\Gamma$ is irreducible and $G$ is not $PSL(2,\mathbb{R})$ 
the result follows from Mostow-Prasad rigidity \cite[Theorem B]{Prasad}.
If $G=PSL(2,\mathbb{R})$ then $\G$ is either cocompact or virtually free,
and we may use multiplicativity of the rational Euler characteristic,
as indicated in the next section.
\end{proof}

{\em Added in proof:} Since writing this article we learned of the existence of \cite{HJ} in which
Humphreys and Johnson study the finite co-Hopf property for lattices.  The first paragraph of the 
preceding proof reduces Proposition \ref{TRHVH} to a special case of their main theorem.

\medskip

Products of irreducible lattices can be dealt with by combining 
the above argument with consideration of the structure of centralisers. 
And in the higher rank case,
far more general results are implied by Margulis super-rigidity \cite{Margulis}. 

Results such as Proposition \ref{TRHVH} allow  one to extend 
the range of maps in classical rigidity theorems. For example:

\begin{corollary} \label{c:mostow}
Let  $M_1$ and $M_2$ be finite volume hyperbolic $n$--manifolds, $n\geq 3$.  
Suppose there are proper open surjective maps $f:M_1\to M_2$ 
and $g:M_2\to M_1$.  
Then $M_1$ and $M_2$ are isometric.   
\end{corollary}
\begin{proof}
The compositions $h= g\circ f :M_1 \to M_1$ and $h'= f\circ g :M_2 \to M_2$ 
are proper, open and surjective. 
Thus by Proposition \ref{Walsh}, 
$h_*\pi_1(M_1)$ has finite index in $\pi_1(M_1)$ and similarly for $h'_{*}$.  
Proposition \ref{TRHVH}(2) tells us
that both $h_*$ and $h'_*$ 
are in fact isomorphisms.   
Hence $f_*$ and $g_*$ are isomorphisms
and Mostow rigidity applies.
\end{proof}
 
Corollary \ref{c:mostow} applies in particular to
hyperbolic knot complements.  Knot complements will be studied  in
more detail in  Section \ref{s:knots}.

\section{Various Hopf properties and volume-type invariants}\label{s:defs}

There are several more notions related to the Hopf property 
that we should consider.
A group $\G$ is {\it finitely co-Hopfian\/} if every injective endomorphism 
with image of finite index is an automorphism.
It is ({\it finitely\/}) {\it hyper-Hopfian\/} if every endomorphism $\varphi$ 
such that $\varphi(\Gamma)$ is a normal subgroup with (finite) cyclic quotient
$\Gamma/\varphi(\Gamma)$ is an automorphism.
Next, we   say that $\Gamma$ satisfies the {\it volume condition\/} 
if isomorphic subgroups of finite index necessarily have the same index.
Finally, we  say that $\Gamma$ satisfies the {\it rank condition\/} 
if any proper subgroup  of finite index requires strictly more generators that $\G$ does.

Only a little argument is needed to see that if $\Gamma$ is cofinitely Hopfian,  
then $\Gamma$ is finitely hyper-Hopfian and finitely co-Hopfian; that 
finitely hyper-Hopfian groups are Hopfian; and that
groups satisfying the volume condition 
or the rank condition are finitely co-Hopfian.
As recalled earlier, residually finite groups are Hopfian.

A useful example to keep in mind is $\IZ_2=\IZ/2\IZ$, 
which is Hopfian and co-Hopfian
but not finitely hyper-Hopfian.

Non-abelian free groups of finite rank and hyperbolic surface groups 
satisfy both the volume condition and the rank condition,
by multiplicativity of the Euler characteristic,
while irreducible lattices
in semi-simple Lie groups satisfy the volume condition but in general
do not satisfy the rank condition (consideration of hyperbolic
3-manifolds that fibre over the circle is enough to prove this).

Less obviously, Reznikov \cite{Re} defined several ``volume-type" invariants of
groups which take values in $[0,\infty)$ and are multiplicative on
passing to subgroups of finite index: if $\G_0<\G$ is a subgroup of
index $d$ then the invariant of $\G_0$ is $d$ times that of $\G$. (These are
called multiplicative invariants in \cite{HJ}.)
He showed that one such invariant,
the ``rank volume" $V_r(\G)$ 
is strictly positive if $\G$ has a presentation of deficiency greater than 1, 
and that $V_r(\G)\geq{V_r(\G/H)}$ for any $H\trianglelefteq\G$.
It follows that every group with deficiency greater than $1$ 
satisfies the volume condition, 
and that it is cofinitely Hopfian if it is Hopfian.

A more classical example of a volume-type invariant
is the {\em{rational Euler characteristic}}.
We recall that if a group $G$ has a subgroup $H$ of finite index 
that has a compact  classifying space
$K(H,1)$, then the rational Euler characteristic
$\chi^{\virt}(\G):=[G:H]^{-1}\chi(H)$ is well defined and multiplicative. 
Borel and Serre \cite{BorelSerre} proved that lattices in semisimple Lie groups have
a well-defined rational Euler characteristic.

Thinking of the rational Euler characteristic in topological terms, one sees
that it is an example of a volume-type invariant that arises in the following
way: one has a class of
spaces (or orbispaces) on which an invariant is defined; each space is defined uniquely
up to some notion of equivalence (e.g. homotopy equivalence or
homeomorphism) by its fundamental group, and the invariant is
constant on equivalence classes; the class is closed under
passage to finite-sheeted covers and the invariant
multiplies  by the index on passage to a finite-sheeted cover. 
 If a group $G$ is torsion-free, residually finite and is
the fundamental group of a space in the class that has non-zero invariant,
 then  $G$ is cofinitely Hopfian.

Gromov's simplicial volume \cite{gromov} fulfills these conditions 
in classes of manifolds that are determined up to homeomorphism
by their fundamental group. 
Classes of manifolds satisfying this condition and for which
this volume is positive include compact locally-symmetric spaces 
of non-compact type and of dimension $\not=4$ \cite{FJ,LS}.
L\"oh and Sauer's notion of Lipschitz simplicial volume is a useful
variant on Gromov's definition: this is non-zero in interesting cases  
where the simplicial volume
vanishes, for example certain of the non-compact locally symmetric
spaces that arise in the context of our Proposition 4.2. In particular, Theorem
1.5 and Corollary 1.6 of \cite{LoS} imply that 
the fundamental group of any locally symmetric space of
non-compact type is cofinitely Hopfian.

\section{Free-by-cyclic groups}\label{s:F.Z}

Finitely presented free-by-cyclic groups have received a great
deal of attention in recent years in part because they form a 
rich context in which to draw out distinctions between the different
notions of non-positive curvature in group theory (cf.~\cite{mb-icm}).
In this section we shall determine which of these groups are
cofinitely Hopfian.

\noindent{\em{Notation:}}
In what follows we let $\Gamma'$, $Z\Gamma$ and $\cd(\Gamma)$
denote the commutator subgroup, centre and cohomological dimension,
respectively, of the group $\Gamma$.
Let $F_r$ denote the free group of rank $r$.

We need the following two lemmas.

\begin{lemma}\label{lem4}
Let $\Gamma={F_r}\rtimes{\IZ}$ and suppose that $\varphi:\Gamma\to{\Gamma}$ 
is an endomorphism such that $N=\varphi(\Gamma)$ is normal in $\Gamma$
and $\Gamma/N\cong{\IZ}$.
Then, either $\Gamma\cong{F_r}\times{\IZ}$ or $\Gamma\cong{\IZ}\rtimes_{-1}\IZ$.
\end{lemma}

\begin{proof} The non-trivial assertion in the lemma
concerns the case $r\ge 2$. 
Thus we assume that $r\ge 2$ and that $\G$ has no decomposition 
$F_\rho\rtimes\IZ$ with $\rho<r$. 
 
The group $\Gamma$ is coherent \cite{FH}, 
and so the finitely generated subgroup $N$ is finitely presentable, 
and  $\cd(N)\leq{\cd(\Gamma)}=2$.
Therefore $\cd(N)=\cd(\Gamma)-\cd(\IZ)=1$, by Theorem 5.6 of \cite{Bi}.
Hence $N$ is also free, of finite rank $s$, say.
Clearly $s\geq{r}$, by minimality of $r$.

Let $p:\Gamma\to\IZ$ be an epimorphism with kernel $F\cong{F_r}$.
The image of $F$ in $N$ is a finitely generated normal subgroup. 
Therefore either $\varphi(F)=1$ or $\varphi(F)$ has finite index in $N$.
If $\varphi(F)=1$ then $N$ is cyclic, contrary to the hypothesis $r>1$.
Hence $[N:\varphi(F)]$ is finite, and so $\varphi(F)$ is free
of finite rank $t$, say.
Now $\chi(\varphi(F))=1-t=[N:\varphi(F)]\chi(N)=[N:\varphi(F)](1-s)\not=0$,
and $t\leq{r}\leq{s}$, since $\varphi(F)$ is a quotient of $F$.
Hence $\varphi(F)=N$ and $t=s$, so $s=r$.
Therefore $\varphi|_F:F\to{N}$ is an isomorphism, by the Hopficity of $F_r$.
Let $q=(\varphi|_F)^{-1}\circ\varphi:\Gamma\to{F}$.
Then $(q,p):\Gamma\to{F}\times{\IZ}$ is an isomorphism. 
\end{proof}

The groups $\Gamma\cong{F_r}\times{\IZ}$ and $\Gamma\cong{\IZ}\rtimes_{-1}\IZ$
are residually finite but are not finitely hyper-Hopfian.
 
\begin{lemma}\label{l:ali}
If $\theta\in {\rm Out}(F_r)$ has infinite order, then the set of integers
$\{n\mid \exists \psi\ {\rm{ with }}\ \psi^n=\theta\}$ is finite.
\end{lemma}

\begin{proof} This is an immediate consequence of a theorem
of Emina Alibegovi\'c \cite{ali}. Her result is cast in the language of
translation lengths. Recall that the translation length $\tau(g)$
of an element $g$ in a group $G$ with finite generating set $A$
is defined to be $\lim_m d(1,g^m)/m$, where $d$ is the word metric
associated to $A$. Since $\tau(g^n)=|n|\tau(g)$ for all $n\in\Z$,
the lemma would be proved if we knew that there was a positive
constant $\epsilon_r>0$ such that $\tau(\theta)>\epsilon_r$ for all  
$\theta\in {\rm Out}(F_r)$ of infinite order. And this is exactly
what Alibegovi\'c proves.
\end{proof}

\begin{theorem}\label{thm5}
Let $\Gamma={F_r}\rtimes_{\theta}{\IZ}$.
Then the following are equivalent:
\begin{enumerate}
\item $\Gamma$ is cofinitely Hopfian;

\item $\Gamma$ is finitely co-Hopfian;

\item  $\Gamma$ is hyper-Hopfian;

\item $\Gamma$ is finitely hyper-Hopfian.

\item $\Gamma$ has trivial centre.
\end{enumerate}
\end{theorem} 

\begin{proof} The theorem is obvious if $r=1$, so we assume that $r\ge 2$  is minimal among 
the ranks of all free normal subgroups of $\G$ with cyclic 
quotient.

The implications $(1)\Rightarrow(2)$, 
$(1)\Rightarrow(4)$ and $(3)\Rightarrow(4)$ are clear, 
while $(2)\Rightarrow(1)$ follows from Lemma \ref{lem2}, 
since $\Gamma$ is finitely generated, residually finite and torsion free.
The implication $(4)\Rightarrow(3)$ follows from Lemma \ref{lem4},
since neither $\Gamma\cong{F_r}\times{\IZ}$ nor $\Gamma\cong{\IZ}\rtimes_{-1}\IZ$
is finitely hyper-Hopfian.

Let $[\theta]$ be the image of $\theta$ in ${\rm Out}(F_r)$
and note that
$Z{\Gamma}=1$  if and only if $[\theta]$ has infinite order.
Suppose that $[\theta]^d=1$ for some finite $d>0$.
If $s\equiv1 \mod d$,  then $[\theta^s]=[\theta]$, 
and so $\Gamma\cong{F_r}\rtimes_{\theta^s}{\IZ}$,
which is isomorphic to a normal subgroup of 
index $s$ in $\Gamma$, with quotient $\IZ/s\IZ$.
Thus $\pi$ is not finitely hyper-Hopfian, and so $(4)\Rightarrow(5)$.

We shall show that $(4)\Rightarrow(2)$ in the course of showing that $(5)\Rightarrow(2)$.
Suppose that $\varphi$ is an injective endomorphism of 
$\Gamma=F_r\rtimes_\theta{\IZ}$ with image $N$ of finite index $d$.
Let $M=\varphi(F_r)$.
Let $\Gamma^\tau<\Gamma$ be the preimage of the torsion subgroup of $\Gamma^{ab}=\Gamma/\Gamma'$.
Then $\Gamma'\leq \Gamma^\tau$ and $\Gamma/\Gamma^\tau\cong \IZ^\beta$.
The image of $N/N^\tau\cong{\IZ}^\beta$ in $\Gamma/\Gamma^\tau$ has finite index, 
since $[\Gamma:N]<\infty$.
Therefore there is an epimorphism $g:\Gamma\to{\IZ}$ such that
$M=N\cap\ker(g)$.
In particular, $M$ has finite index in $\ker(g)$.
Hence $M=\ker(g)$, by minimality of $r$,
and so $N$ is normal in $\Gamma$ and $\Gamma/N\cong{\IZ/d\IZ}$.
Thus $(4)\Rightarrow(2)$, and we have
proved that the first four conditions are equivalent.

Continuing with the notation of the previous paragraph,
we have $\G/M\cong\Z$. Fix $t\in\Gamma$ such that $g(t)$ generates $G/M$. 
Conjugation by $t$ induces an automorphism $\psi\in{\rm{Aut}}(M)$. 
Then $N$ is generated by $M$ and $t^d$ and $\varphi$ induces 
an isomorphism from $\G=F_r\rtimes_\theta\Z$ to $N=M\rtimes_{\psi^d}\Z$ 
that is compatible with the given semidirect product decompositions. 
It follows that the image of $\phi$ is conjugate to the image of $\psi^d$ in ${\rm Out}(F_r)$. 
By considering iterates of $\varphi$, we conclude that the image of $\theta$ 
has a $d^n$-root in ${\rm Out}(F_r)$ for all positive integers $n$. 
By Lemma \ref{l:ali}, it follows that $\theta$ has finite order.
Thus $(5)\Rightarrow(2)$. 
This completes the proof.
\end{proof}

If $\Gamma/\Gamma^\tau\cong\IZ$ then the 
preceding proof can be simplified. First, the implication $(4)\Rightarrow(1)$ follows from 
the observation that if $\Gamma=N\rtimes_\theta{\IZ}$ where $N$ is cofinitely Hopfian and 
$\varphi:\Gamma\to{\Gamma}$ is an endomorphism with image of finite index and $\varphi(N)\leq{N}$, 
then $\varphi$ is a monomorphism onto a normal subgroup and $\Gamma/\varphi(\Gamma)$ is finite cyclic.
And in proving  $(5)\Rightarrow(2)$ one does not need Lemma \ref{l:ali}.

\begin{remark}\label{r:coH}
The  conditions listed in the preceding
theorem do not imply that $\Gamma$ is co-Hopfian.
For instance, if $K$ is a composite of fibred knots or 
is an algebraic knot other than a torus knot and 
$\Gamma$ is the fundamental group of $\S^3\smallsetminus K$,  
then $\Gamma/\Gamma'\cong{\IZ}$, $\Gamma'$ is free and $Z{\Gamma}=1$
but $\Gamma$ is not co-Hopfian \cite{GW}.
\end{remark}

Let $G=F*_\theta$ be an ascending HNN extension, where $F=F_r$ for some $r>1$ and
$\theta:F\to{F}$ is a monomorphism with image a proper subgroup.
Then $G$ has trivial centre and is Hopfian \cite{GMSW}. 
Is it cofinitely Hopfian?
Lemma \ref{lem4} may be adapted to exclude such groups,
while the implications $(1)\Leftrightarrow(2)$, 
$(1)\Rightarrow(4)$ and $(3)\Rightarrow(4)$ of Theorem \ref{thm5} hold.
Does $(4)\Rightarrow(2)$?
(If $r=1$ then $G$ is a solvable Baumslag-Solitar group, 
which is not finitely hyper-Hopfian.)

\section{3-Manifolds and Knot groups}\label{s:knots}

Let $M$ be a compact orientable 3-manifold and let $\pi=\pi_1(M)$. 
Then $M$ is Haken, hyperbolic or Seifert fibred,
as a consequence of Thurston's Geometrisation Conjecture,
now proven by Perelman.
In each case $\pi$ is residually-finite and hence Hopfian \cite{He,Th}.
If $M$ is closed then $\pi$ is co-Hopfian if and only if $M$ 
is irreducible and has no finite cover which is a direct product 
of a (closed) surface with $\S^1$ or a torus bundle over $\S^1$ \cite{WW}.
If $M$ is irreducible and has nonempty toral boundary then $\pi$ is co-Hopfian 
if and only if it is not $\IZ^2$ and no non-trivial Seifert fibred piece 
of the JSJ decomposition of $M$ meets $\partial{M}$, 
by Theorem 2.5 of \cite{GW}.
The group $\pi$ satisfies the volume condition 
(as defined in Section 5 above) if and only if 
it is either a proper free product other than $\Z/2*\Z/2$
or if $M$ is irreducible and has no finite cover which is a direct product 
of a surface with $\S^1$ or a torus bundle over $\S^1$ \cite{WW,WY}.

From these results we see:
if $K$ is a knot in $\IS^3$, then the knot group $\pi_K$ is Hopfian;
it is co-Hopfian if and only if $K$ is not a torus knot, 
cable knot or composite knot; and it satisfies the volume condition
(and is finitely co-Hopfian) if and only if $K$ is not a torus knot.
(See Corollaries 2.6 and 7.5 of \cite{GW}, \cite{WY} and subsection \ref{sst}
below.)  The following theorem provides the complementary
classification for the cofinite Hopf property.

We note also the following standard facts from knot theory.
Let $K$ be a knot with group $\pi=\pi_K$.
Then the abelianization $\pi/\pi'$ is infinite cyclic.
If $K$ is non-trivial then the image of the fundamental group of a 
Seifert surface of minimal genus is a non-abelian free subgroup of $\pi$.
In particular, the group of a fibred knot is free-by-cyclic.
Torus knots are fibred and their groups have infinite cyclic centres.
(See \cite{Ro}, for example.)

\begin{theorem}\label{thm6}
Let $K$ be a knot in $\S^3$.
The knot group $\pi_{K}$ is cofinitely Hopfian 
if and only if $K$ is not a torus knot.
\end{theorem}

\begin{proof}
If $K$ is a torus knot then $\pi_{K}'$ is a finitely generated free group and
$Z\pi_{K}\not=1$.
Therefore $\pi_{K}$ is not cofinitely Hopfian, by Theorem \ref{thm5}.
(A more explicit proof is given below.)

Conversely, if $K$ is not a torus knot then $\pi_{K}$ is finitely co-Hopfian
\cite{GW}.
Since $\pi_{K}$ is finitely generated, residually finite and torsion free
Lemma \ref{lem2} then implies that $\pi_{K}$ is cofinitely Hopfian.
\end{proof}

We recover the following result of \cite{Si}.

\begin{corollary} 
Let $K$ be a knot in $\S^3$.
Then $\pi_{K}$ is hyper-Hopfian if and only if 
$K$ is not a torus knot.
\end{corollary}

\begin{proof} 
Let $\pi=\pi_{K}$ and assume that $K$ is not a torus knot.
Then $K$ is non-trivial and $\pi$ is non-abelian.
Let $\varphi:\pi\to\pi$ be an endomorphism with $\varphi(\pi)$ 
a normal subgroup of $\pi$ and $\pi/\varphi(\pi)$ cyclic.
Then $\pi'\leq\varphi(\pi)$, since knot groups have abelianization $\IZ$.
If $\varphi(\pi)=\pi'$,  then $\pi'$ is finitely generated.
But then $K$ is fibred and so $\pi'$ is free \cite{St}.
Since $\pi$ and $\varphi(\pi)$ have cyclic abelianization
this is only possible if $\pi\cong{\IZ}$, 
which is contrary to our assumption.
Therefore $\pi/\varphi(\pi)$ is finite, 
and so $\varphi$ is an automorphism by Theorem \ref{thm6}.
\end{proof}

The complement of any knot in $\S^3$ is aspherical \cite{papa} 
and hence is determined up to homotopy equivalence by its fundamental group. 
But, famously, the fundamental group does not determine the knot 
exterior up to homeomorphism.
For example, the granny knot 
(the sum of two copies of the left hand trefoil knot) 
and the reef knot (the sum of a trefoil knot and its reflection) 
have the same group, but their exteriors are not homeomorphic.
On the other hand,  Waldhausen \cite{Wald} proved that any
homotopy equivalence of knot complements that preserves the {\em{peripheral structure}} 
is homotopic to a homeomorphism.  
More precisely, he shows that if $X$ is the complement 
of an open regular neighbourhood of a knot $K$, 
then any homotopy equivalence of the pair $(X,\partial X)$ 
is homotopic to a homeomorphism 
(and it is easy to extend this to the whole of $\S^3\smallsetminus K$).

\begin{theorem}\label{corabove}
Let $K$ be a knot that is not a torus knot and let $f$
be a proper open self-map of the knot complement
${\IS^3\smallsetminus K}$. Then $f$ is homotopic to a homeomorphism.
\end{theorem}

\begin{proof}
Since $f$ is proper and open and $\IS^3\smallsetminus K$ is connected, 
$f$ is surjective.
Therefore $n=[\pi_{K}:f_*(\pi_{K})]$ is finite, by Proposition \ref{Walsh}.
Since $K$ is not a torus knot, 
Theorem \ref{thm6} tells us that $f_*$ is an isomorphism. 
And since knot complements are aspherical, 
it follows that $f$ is a homotopy equivalence. 
By the work of Waldhausen quoted above, we will be done
provided that we can argue that $f$ preserves the peripheral structure of the knot.
For the benefit of readers who are not specialists in 3-manifold theory, we 
explain why this is true  with a proof that has a group-theoretic flavour.

Let $N\cong{\S^1\times{\D^2}}$ be a closed regular neighbourhood of $K$,
where $\D$ is the open unit disc in $\R^2$. Let $r\D$ be the
disc of radius $r\in (0,1]$, let $rN$ be the corresponding neighourhood of $K$
and let $U_r=rN\smallsetminus{K}$.
The closure $X$ of $\S^3\smallsetminus{N}$ is a compact manifold 
with boundary $\partial{X}$
a torus, and $X\hookrightarrow\S^3\smallsetminus{K}$ is a homotopy equivalence.
Adjusting $f$ by a homotopy if necessary, 
we may assume that $f(X)\subseteq{X}$. 
Moreover, since $f$ is proper, 
we may assume that it maps each $U_r$ inside a suitable
$U_{r'}$ (since the sets $U_r$ form a cofinal system of closed neighbourhoods 
of the single end of the noncompact $3$-manifold $\S^3\smallsetminus{K}$).
Therefore, adjusting by a radial projection, 
we can arrange that $f(\partial X)\subset \partial X$.
We fix a basepoint $x_0\in \partial X$ and homotope $f$ so that $f(x_0)=x_0$.

Since $K$ is non-trivial,
$\pi_1(\partial X, x_0)\to\pi_{K}=\pi_1(\S^3\smallsetminus K, x_0)$ 
is injective (cf.~\cite{papa}); let $P$ denote its image.
We will be done if we
can prove that $f|_{\partial X}$ induces an isomorphism $P\to P$.
 
Since $f_*$ is injective,  $f_*^{-1}(P)\cong{\Z^2}$, 
so if $f_*(P)$ were not the whole of $P$ then there would be elements 
of $\pi_K\smallsetminus P$ that commuted with $P$. 
But then $\G=\pi_{K}*_P\pi_{K}$ would contain a copy of $\Z^3$, 
and $\G$ is the fundamental group of the closed, 
aspherical $3$-manifold $Y$ obtained by doubling $X$ along its boundary. 
Since infinite coverings of $3$-manifolds collapse to 2-complexes, $\Z^3$
cannot be the fundamental group of an infinite-sheeted covering of $Y$, so
$\pi_1(Y)$ would be virtually abelian.
But this is absurd, since $\pi_1(Y)$ contains $\pi_K$ 
and non-trivial knot groups have non-abelian free subgroups.
This contradiction completes the proof.
\end{proof}
       
\begin{corollary}
Let $K$ be a knot and let $f:\IS^3\to\IS^3$ 
be an open mapping of finite degree not equal to 1 such that $f^{-1}(K)=K$. 
Then $K$ is a torus knot.
\end{corollary}

\begin{proof}
The map $f$ is surjective, since it is open and $\IS^3$ is compact.
Therefore $f(K)=K$ and the restriction  
$f:\IS^3\smallsetminus K \to \IS^3\smallsetminus K$ is proper, 
open and surjective.  
Since the degree of $f$ is not $1$,  
$f$ cannot be homotopic to a homeomorphism.  
Thus Theorem \ref{corabove} implies $K$ is a torus knot.
\end{proof}
 
A map is said to be {\em{discrete}}
if the preimage of each point in the target is a discrete set.
 
\begin{corollary}
Let $K$ be a non-torus knot and let $f:\IS^3\to\IS^3$ be a discrete open mapping such that $f^{-1}(K)=K$. Then $f$ is a homeomorphism.
\end{corollary}

\begin{proof}
A discrete open mapping homotopic to a homeomorphism is a homeomorphism 
-- see \cite{Cern,Vai}.
\end{proof}

Notice, with regard to the context of  \cite{BHM}, that
the above corollary implies that the branch set of a quasiregular map 
cannot be a  completely invariant non-torus knot.  
The assumption here that the mapping is discrete 
(as opposed to {\em light}) is necessary due to the counterexamples 
constructed in \cite{Wilson}.

\subsection{Torus knots are not cofinitely Hopfian}\label{sst}

The fundamental group of a torus knot has the form
$$\G_{m,n}=\langle a, b\mid a^n=b^m \rangle .$$
The centre of such a group is infinite cyclic, generated
by $z:=a^n=b^m$. Let $Q\cong (\Z/n)\ast(\Z/m)$
be the quotient by the centre.

We consider maps  $\varphi:\G_{m,n}
\to \G_{m,n}$ of the form
\[ \varphi(a)=az^p\;\;\;\; {\rm and} \;\;\;\; \varphi(b)=bz^q.\] 
This formula defines a homomorphism if and only if
 $\varphi(a^n) = \varphi(b^m)$, that is, $np=mq$ (which we write as $r-1$). 
 The image of $\varphi$  
is of finite index because it maps onto $Q$ and
intersects $Z=\langle z\rangle$ non-trivially. 
 The normal form theorem for amalgamated free
products shows that $\varphi$ is injective. Since no
commutators lie in the kernel,
$\varphi^{-1}(Z) = Z$. And
since  $\varphi(z) = \varphi(a^n)=z^{r}$, we see that
$\varphi$ is not onto unless $p=q=0$. Thus we have proved:

\begin{proposition} If $\G_{m,n}=\langle a, b\mid a^n=b^m \rangle $
and $r$ is an integer, with $r\equiv 1\mod n$ and $r\equiv 1\mod m$,
then $[\varphi(a):=a^{r}$,  $\varphi(b):=b^{r}]$ defines a monomorphism
$\varphi:\G_{m,n}\to \G_{m,n}$ whose image has finite index. But
$\varphi$ is onto only if $r=\pm 1$.
\end{proposition}

\bigskip

\noindent M. R.  Bridson - University of Oxford, UK.

bridson@maths.ox.ac.uk

\noindent D. Groves - University of Illinois at Chicago,  USA.

groves@math.uic.edu

 \noindent J. Hillman - The University of Sydney,  Sydney, Australia.
 
jonathan.hillman@sydney.edu.au

\noindent G.J. Martin -  Massey University,  Auckland, NZ.

g.j.martin@massey.ac.nz

\end{document}